\documentclass[11pt]{article}

\usepackage{amssymb} 
\usepackage{amsthm} 
\usepackage{latexsym}
\usepackage{amsmath} 
\usepackage[utf8]{inputenc} 
\usepackage[dvips]{color}

\usepackage{multicol}
\usepackage{graphicx}

 \let\oldmarginpar\marginpar
\renewcommand\marginpar[1]{\-\oldmarginpar[\raggedleft\footnotesize #1]%
  {\raggedright\footnotesize #1}}

\def\n{\nabla}

\def\F{{\mathcal F}}
\def\P{{\mathcal P}}
\def\I{{\mathcal I}}

\theoremstyle{plain} 
\newtheorem{theorem}{Theorem}[section]
\newtheorem{proposition}[theorem]{Proposition}

\newtheorem{corollary}[theorem]{Corollary}
\theoremstyle{definition}

\newtheorem{lemma}[theorem]{Lemma}
\newtheorem{remark}[theorem]{Remark}
\theoremstyle{remark}

\newcommand{\bp}{\begin{proof}\;}
  \newcommand{\ep}{\end{proof}}

\title{Characterization of projectively flat Finsler manifolds of
  constant curvature with finite dimensional holonomy group}

\author{Zolt\'an Muzsnay and P\'eter T. Nagy}

\date{}

\begin{document}

\maketitle

\begin{center}
  \emph{Dedicated to Professor Lajos Tamássy on his 90th birthday}
\end{center}

\begin{abstract}

  \noindent
  In this paper we prove that the holonomy group of a simply connected locally
  projectively flat Finsler manifold of constant curvature is a finite
  dimensional Lie group if and only if it is flat or it is Riemannian.
\end{abstract}



\footnotetext{This research was supported by the Hungarian Scientific Research
  Fund (OTKA) Grant K 67617.}

\section{Introduction}

A Finsler manifold is a pair $(M, F)$, where $M$ is an $n$-manifold and
$F\!:TM \to \mathbb R$ is a non-negative function, smooth and positive away
from the zero section of $TM$, positively homogeneous of degree 1, and
strictly convex on each tangent space.  A Finsler manifold of dimension $2$ is
called \emph{Finsler surface}.
\\[1ex]
The concept of Finsler manifold is a direct generalization of the Riemannian
one.  The fundamental tensor $g=g_{ij}dx^i\otimes dx^j$ associated to $\F$ is
formally analogous to the metric tensor in Riemannian geometry.  It is defined
by
\begin{equation} 
  \label{eq:g}
  g_{ij} := \frac{1}{2}\frac{\partial^2 \F^2}{ \partial {y^i} \partial
    {y^j}},
\end{equation} 
in an induced standard coordinate system $(x,y)$ on $TM$. As in Riemannian
geometry, a canonical connection $\Gamma$ can be defined for a Finsler space
\cite{Gri}.  However, since the energy function $E=\frac{1}{2}\F^2$ is not
necessarily quadratic and only homogeneous, the connection is in general
non-linear. In the case, when the connection $\Gamma$ is linear, the Finsler
space is called \textit{Berwald space}. In particular, every Riemannian
manifold is a Berwald space.
\\[1ex]
Due to the existence of the canonical connection, the holonomy group of a
Riemannian or Finsler manifold can be defined in a very natural way: it is the
group generated by parallel translations along closed curves.  In the
Riemannian case, since the Levi-Civita connection is linear and preserve the
Riemannian metric, the holonomy group is a Lie subgroup of the orthogonal
group $O(n)$ (see \cite{BL}).  The Riemannian holonomy theory has been
extensively studied, and by now, its complete classification is known.
\\[1ex]
The holonomy properties of Finsler spaces is essentially different from the
Riemannian one.  It is proved in \cite{Mu_Na} that the holonomy group of a
Finsler manifold of nonzero constant curvature with dimension greater than $2$
is not a compact Lie group. In \cite{MuNa1} large families of projectively
flat Finsler manifolds of constant curvature are constructed such that their
holonomy groups are not finite dimensional Lie groups.  There are explicitly
given examples of Finsler 2-manifolds having maximal holonomy group. In these
examples the closure of the holonomy group is isomorphic to the orientation
preserving diffeomorphism group of the 1-dimensional sphere \cite{MuNa3}.
\\[1ex]
In this paper we are investigating the holonomy group of locally projectively
flat Finsler manifolds of constant curvature.  A Finsler function $\F$ on an
open subset $D \subset \mathbb R^n$ is called \emph{projectively flat}, if all
geodesic curves are straight lines in $D$. The Finsler manifold $(M,\F)$ is
said to be locally projectively flat, if for any point there exists a local
coordinate system in which $\F$ is projectively flat.  Our aim is to
characterize all locally projectively flat Finsler manifolds with finite
dimensional holonomy group.  To obtain such a characterization, we will
investigate the dimension of the infinitesimal holonomy algebra which was
introduced in the Finsler case in \cite{MuNa}.  In Proposition \ref{sec:thm1}
we prove that if $(M,\mathcal F)$ is a non-Riemannian locally projectively
flat Finsler manifolds of nonzero constant curvature, then its infinitesimal
holonomy algebra is infinite dimensional.  Using this result and the tangent
property of the infinitesimal holonomy algebra proved in \cite{MuNa} we obtain
the characterization given by Theorem \ref{sec:thm2}: \emph{The holonomy group
  of a locally projectively flat Finsler manifold of constant curvature is
  finite dimensional if and only if it is a Riemannian manifold or a flat
  Finsler manifold.}

\section{Preliminaries}

Throughout this article, $M$ is a $C^\infty$ smooth simply connected manifold,
${\mathfrak X}^{\infty}(M)$ is the vector space of smooth vector fields on $M$
and ${\mathsf {Diff}}^\infty(M)$ is the group of all $C^\infty$-diffeomorphism
of $M$.  The first and the second tangent bundles of $M$ are denoted by
$(TM,\pi ,M)$ and $(TTM,\tau ,TM)$, respectively.

\subsection{Finsler manifolds}

A \emph{Finsler manifold} is a pair $(M,\mathcal F)$, where the Finsler function 
$\F\colon TM\to \mathbb{R}$ is a continuous function, smooth on $\hat T M \!:=
\!TM\!  \setminus\! \{0\}$, its restriction ${\mathcal F}_x={\mathcal
  F}|_{_{T_xM}}$ is a positively homogeneous function of degree one and
the symmetric bilinear form
\begin{equation}\label{fins}
  g_{x,y} \colon (u,v)\ \mapsto \ g_{ij}(x, y)u^iv^j=\frac{1}{2}
  \frac{\partial^2 \mathcal F^2_x(y+su+tv)}{\partial s\,\partial t}\Big|_{t=s=0}
\end{equation}
is positive definit. The Finsler manifold $(M,\mathcal F)$ is Riemannian, if
$\mathcal F^2$ induces a quadratic form on any tangent space $T_xM$. Hence we
say that $(M,\mathcal F)$ is \emph{non-Riemannian Finsler manifold} if there
exists a point $x\in M$ such that $\mathcal F_x^2$ is not quadratic.
\\[1ex]
A vector field $X(t)=X^i(t)\frac{\partial}{\partial x^i}$ along a curve $c(t)$
is said to be parallel with respect to the associated \emph{homogeneous
  (nonlinear) connection} if it satisfies
\begin{equation}
  \label{eq:D}
  D_{\dot c} X (t):=\Big(\frac{d X^i(t)}{d t}+  G^i_j(c(t),X(t))\dot c^j(t)
  \Big)\frac{\partial}{\partial x^i} =0,
\end{equation}
where the \emph{geodesic coefficients} $G^i(x,y)$ are given by
\begin{equation}
  \label{eq:G_i}  G^i(x,y):= \frac{1}{4}g^{il}(x,y)\Big(2\frac{\partial
    g_{jl}}{\partial x^k}(x,y) -\frac{\partial g_{jk}}{\partial
    x^l}(x,y) \Big) y^jy^k.
\end{equation}
and  $G^i_j=\frac{\partial G^i}{\partial y^j}$. 
The \emph{horizontal Berwald covariant derivative} $\nabla_X\xi$ of
$\xi(x,y) = \xi^i(x,y)\frac {\partial}{\partial y^i}$ by the vector
field $X(x) = X^i(x)\frac {\partial}{\partial x^i}$ is expressed locally
by
\begin{equation}
  \label{covder}
  \nabla_X\xi = \left(\frac {\partial\xi^i(x,y)}{\partial x^j} 
    - G_j^k(x,y)\frac{\partial \xi^i(x,y)}{\partial y^k} + 
    G^i_{j k}(x,y)\xi^k(x,y)\right)X^j\frac {\partial}{\partial y^i}, 
\end{equation}
where $G^i_{j k}(x,y) := \frac{\partial G_j^i(x,y)}{\partial
  y^k}$.
\\[2ex]
A Finsler manifold $(M, \F)$ is said to be \emph{projectively flat}, if there 
exists a diffeomorphism of $M$ to an open subset $D \subset \mathbb R^n$ such 
that the images of geodesic curves are straight lines in $D$. A Finsler 
manifold $(M, \F)$ is said to be \emph{locally projectively flat}, if for 
every $x\in M$ there exists a local coordinate 
system $(U,x)$ such that $x=(x^1,\dots ,x^n)$ is mapping the
neighbourhood $U$ into the Euclidean space $\mathbb R^n$ such that the
straight lines of $\mathbb R^n$ correspond to the geodesics of $(M, \F)$ on
$U$. Then there exists a function $\P(x,y)$, such that the geodesic
coefficients are given by
\begin{equation}
  \label{eq:proj_flat_G_i_1}
  \hphantom{\quad \quad i=1,...,n} G^i(x,y) = \P(x,y)y^i, \quad\quad i=1,...,n
\end{equation}
The function $\P\!=\!\P(x,y)$ is called the \emph{projective factor} of
$(M,\F)$ on $U$. Since it is 1-homogeneous in the $y$-variable, we have also
the following relations:
\begin{equation}
  \label{eq:proj_flat_G_i_2}
  G^i_k = \frac{\partial\P}{\partial y^k}y^i + \P\delta^i_k,\quad G^i_{kl} 
  = \frac{\partial^2\P}{\partial y^k\partial y^l}y^i 
  + \frac{\partial \P}{\partial y^k}\delta^i_l + 
  \frac{\partial \P}{\partial y^l}\delta^i_k. 
\end{equation}
It follows from equation (\ref{eq:proj_flat_G_i_1}) that the associated
homogeneous connection (\ref{eq:D}) is linear if and only if the projective
factor $\P(x,y)$ is linear in $y$.  According to Lemma 8.2.1 in \cite{ChSh}
p.155, if $(M \!  \subset\! \mathbb R^n, \F)$ is a projectively flat Finsler
manifold with constant flag curvature $\lambda$, then we have
\begin{equation}
  \label{eq:P}
  \P = \frac{1}{2\F}\frac{\partial\F}{\partial x^i}y^i, \qquad 
  \quad  \P^2 -\frac{\partial\P}{\partial x^i}y^i 
  = \lambda \F^2.
\end{equation}
Hence if $\lambda\neq 0$ and $\P(x,y)$ is linear in $y$ at $x\in M$ then
$\F^2(x,y)$ is a quadratic form in $y$ at $x$.

\subsection{Holonomy group, curvature, infinitesimal holonomy algebra}

For a Finsler manifold $(M,\mathcal F)$ of dimension $n$ the \emph{indicatrix}
at $x \in M$ is
\begin{displaymath}
  \I_xM:= \{y \in T_xM \ | \ \ \F(y) = 1\}
\end{displaymath}
in $T_xM$ which is an $(n - 1)$-dimensional submanifold of $T_xM$.  We denote
by $(\I M,\pi,M)$ the \emph {indicatrix bundle} of $(M,\mathcal F)$.  We
remark that, although the homogeneous (nonlinear) parallel translation is in
general not metrical, that is it does not preserve the Finsler metric tensor (\ref{fins}), but
it preserves the value of the Finsler function.  That means that for any
curves $c:[0,1]\to M$, the induced parallel translation $\tau_{c}:T_{c(0)}M\to
T_{c(1)}M$ induces a map $\tau_{c}\colon \I_{c(0)}M\rightarrow \I_{c(1)}M$
between the indicatrices.
\\[1ex]
The \emph{holonomy group} $\mathsf{Hol}_x(M)$ of $(M,\F)$ at a point $x\in M$
is the subgroup of the group of diffeomorphisms
${\mathsf{Diff}^{\infty}}({\I}_xM)$ generated by (nonlinear) parallel
translations of ${\I}_xM$ along piece-wise differentiable closed curves
initiated at the point $x\in M$.
\\[2ex]
The \emph{Riemannian curvature tensor}
\begin{math}
  R_{}\!= \! R^i_{jk}(x,y) dx^j\otimes dx^k \otimes
  \frac{\partial}{\partial x^i}
\end{math}
has the expression
\begin{displaymath}
  R^i_{jk}(x,y) =  \frac{\partial G^i_j(x,y)}{\partial x^k} 
  - \frac{\partial G^i_k(x,y)}{\partial x^j} + 
  G_j^m(x,y)G^i_{k m}(x,y) - G_k^m(x,y)G^i_{j m}(x,y). 
\end{displaymath} 
The manifold has \emph{constant flag curvature} $\lambda\in{\mathbb R}$, if
for any $x\in M$ the local expression of the Riemannian curvature is
\begin{displaymath}
  R^i_{jk}(x,y) = \lambda\big(\delta_k^ig_{jm}(x,y)y^m -
  \delta_j^ig_{km}(x,y)y^m\big).
\end{displaymath}
For any vector fields $X, Y\in {\mathfrak X}^{\infty}(M)$ on $M$ the vector
field $\xi = R(X,Y)\in {\mathfrak X}^{\infty}({\I}M)$ is called a
\emph{curvature vector field} of $(M, \F)$ (see \cite{Mu_Na}). The Lie algebra
$\mathfrak{R}(M)$ of vector fields generated by the curvature vector fields of
$(M, \F)$ is called the \emph{curvature algebra} of $(M, \F)$.  The
restriction $\mathfrak{R}_x(M) \! := \!  \big\{\,\xi\big|_{\I_xM}\ ;\
\xi\in\mathfrak{R}(M)\, \big\} \subset {\mathfrak X}^{\infty}({\I}_xM)$ of the
curvature algebra to an indicatrix $\I_x M$ is called the \emph{curvature
  algebra at the point $x\in M$}.
\\
We remark that the indicatrix of a Finsler surface is 1-dimensional at any
point $x\in M$, hence the curvature vector fields at $x\in M$ are proportional
to any given non-vanishing curvature vector field. Therefore the curvature
algebra $\mathfrak{R}_x(M)$ is at most a 1-dimensional commutative Lie
algebra.
\\[2ex]
The \emph{infinitesimal holonomy algebra} of $(M,\mathcal F)$ is the smallest
Lie algebra $\mathfrak{hol}^{*}(M)$ of vector fields on the indicatrix bundle
$\I M$ containing the curvature algebra and invariant with respedct to the horizontal
Berwald covariant differentiation. The $\mathfrak{hol}^{*}(M)$ is
characterized by the following properties: \vspace{-3pt}
\begin{enumerate}
\item[(i)] any curvature vector field $\xi$ belongs to
  $\mathfrak{hol}^{*}(M)$, \vspace{-3pt}
\item[(ii)] if $\xi, \eta\in \mathfrak{hol}^{*}(M)$ then
  $[\xi,\eta]\in\mathfrak{hol}^{*}(M)$, \vspace{-3pt}
\item[(iii)] if $\xi\in\mathfrak{hol}^{*}(M)$ and $X\in {\mathfrak
    X}^{\infty}(M)$ then $\nabla_{\!\! X}\xi \in
  \mathfrak{hol}^{*}(M)$. \vspace{-2pt}
\end{enumerate} 
The restriction
\begin{math}
  \mathfrak{hol}^{*}_x(M) \! := \!  \big\{\,\xi\big|_{\I_xM}\ ;\
  \xi\in\mathfrak{hol}^{*}(M)\, \big\} \subset {\mathfrak
    X}^{\infty}({\I}_xM)
\end{math}
of the infinitesimal holonomy algebra to an indicatrix $\I_x M$ is called the
\emph{infinitesimal holonomy algebra at the point $x\in M$}.
\\
Clearly, we have $\mathfrak{R}(M)\subset\mathfrak{hol}^{*}(M)$ and
$\mathfrak{R}_x(M)\subset\mathfrak{hol}^{*}_x(M)$ for any $x\in M$ (see
\cite{MuNa}).

\section{Dimension of the holonomy group}

Let $(M,F)$ be a Finsler manifold and $x\in M$ an arbitrary point in $M$.
According to Proposition 3 of \cite{MuNa}, the infinitesimal holonomy algebra
$\mathfrak{hol}^{*}_x(M)$ is tangent to the holonomy group
$\mathsf{Hol}_x(M)$. Therefore the group generated by the exponential image of
the infinitesimal holonomy algebra at $x\in M$ with respect to the exponential
map
\begin{math}
  \exp_x\! :\!{\mathfrak X}^{\infty}({\I}_xM) \to
  {\mathsf{Diff}^{\infty}}({\I}_xM)
\end{math}
is a subgroup of the closed holonomy group $\overline{\mathsf{Hol}_x(M)}$ (see
Theorem 3.1 of \cite{MuNa3}). Consequently, we have the following estimation
on the dimensions:
\begin{equation}
  \label{eq:dim}
  \dim\mathfrak{hol}^{*}_x(M) \leq \dim \mathsf{Hol}_x(M).
\end{equation}
Using the result of S.~Lie claiming that the dimension of a finite-dimensional
Lie algebra of vector fields on a connected $1$-dimensional manifold is less
than $4$ (cf.~\cite{Lie}, Theorem 4.3.4) we can obtain the following
\begin{lemma}
  \label{simultan}
  If the infinitesimal holonomy algebra $\mathfrak{hol}^{*}_x(M)$ of a Finsler
  surface $(M,\F)$ contains $4$ simultaneously non-vanishing $\mathbb
  R$-linearly independent vector fields, then $\mathfrak{hol}^{*}_x(M)$ is
  infinite dimensional.
\end{lemma}
\begin{proof}
  If the infinitesimal holonomy algebra is finite-dimensional, then the 
  dimension of the corresponding Lie group acting locally effectively on the
  1-dimensional indicatrix would be at least 4, which is a contradiction.
\end{proof}
\noindent
Using Lemma \ref{simultan} we can prove the following 
\begin{proposition} 
  \label{sec:thm1}
  The infinitesimal holonomy algebra of any locally projectively flat 
  non-Riemannian Finsler surface $(M,\mathcal F)$ of constant curvature
  $\lambda\neq 0$ is infinite dimensional.
\end{proposition}
\begin{proof}
  Assume that the locally projectively flat Finsler surface $(M, \F)$ of
  non-zero constant curvature $\lambda$ is non-Riemannian at a fixed point
  $x\!\in\! M$. Let $(x^1,x^2)$ be a local coordinate system centered at $x$,
  corresponding to the canonical coordinates of the Euclidean plane which is
  projectively related to $(M, \F)$, and let $(y^1,y^2)$ be the induced
  coordinate system in the tangent planes $T_xM$.
  \\[1ex]
  Consider the curvature vector field
  \begin{displaymath}
    \xi(x,y)\!=\!R \left( \frac{\partial}{\partial x_1}, 
      \frac{\partial}{\partial x_2} \right)(x,y) 
    = \lambda\big(\delta_2^ig_{1m}(x,y)y^m -\delta_1^ig_{2m}(x,y)y^m\big) 
    \frac{\partial}{\partial x^i} 
  \end{displaymath}
  at the point $x\!\in\! M$.  Since $(M, \F)$ is of constant flag curvature,
  the horizontal Berwald covariant derivative $\nabla_WR$ of the tensor field
  $R$ vanishes and one has
  \begin{displaymath}
    \n_W\xi = R\left(\!\n_k \!\left(\!\frac{\partial}{\partial x^1}
        \wedge\frac{\partial}{\partial x^2}\right)\!\right)W^k.
  \end{displaymath}
  Since
  \begin{displaymath}
    \n_k\left(\frac{\partial}{\partial x^1} \wedge 
      \frac{\partial}{\partial x^2}\right) = 
    \left(G^1_{k1} + G^2_{k2}\right)\frac{\partial}{\partial x^1}
    \wedge\frac{\partial}{\partial x^2}
  \end{displaymath}
  we obtain $\n_W\xi = \left(G^1_{k1} + G^2_{k2}\right)W^k\xi$. According to
  (\ref{eq:proj_flat_G_i_2}) we have
  \begin{math}
    G^m_{km} = 3\frac{\partial P}{\partial y^k}
  \end{math}
  and hence $\nabla_k \xi = 3\frac{\partial P}{\partial y^k}\xi$, where
  $\nabla_k = \nabla_{\!\!\frac{\partial}{\partial x^k}}$.  Moreover we have
  \begin{displaymath}
    \nabla_j \!\left(\!\frac{\partial \P}{\partial y^{k}}\!\right)
    = \frac {\partial^2 \P}{\partial x^{j}\partial y^{k}} 
    - G_j^{m}\frac{\partial^2 \P}{\partial y^{m}\partial y^{k}}
    = \frac{\partial^2 \P}{\partial x^{j}\partial y^{k}} -
    \P\frac{\partial^2 \P}{\partial y^{k}\partial
      y^{j}},
  \end{displaymath}
  and hence 
  \begin{displaymath}
    \n_j\!\left(\n_k\xi \right) = 
    3\left(\frac {\partial^2 \P}{\partial x^{j}\partial y^{k}} - 
      \P\frac{\partial^2 \P}{\partial y^{k}\partial y^{j}} + 3\frac{\partial \P}
      {\partial y^{k}} \frac{\partial \P}{\partial y^{j}}\right) \xi.
  \end{displaymath}
  According to Lemma 8.2.1, equation (8.25) in \cite{ChSh}, p. 155, we have
  \begin{equation}
    \label{rapcsak} 
    \frac{\partial^2 \P}{\partial x^j\partial y^k} =     
    \frac{\partial \P}{\partial y^j}\frac{\partial \P}{\partial y^k}      
    +  \frac{\partial^2 \P}{\partial y^j\partial y^k} - \lambda\,g_{jk},
  \end{equation}
  hence
  \begin{math}
    \n_j\!\left(\n_k\xi\right) = 3\left(\!4 \frac{\partial \P}{\partial y^j}
      \frac{\partial \P}{\partial y^k} - \lambda\,g_{jk} \right) \xi.
  \end{math}
  It follows the
  \begin{lemma}
    \label{sec:vect_func}
    For any fixed $1\leq j,k \leq 2$
    \begin{equation}
      \label{eq:4_vectors}
      y\to \xi(x,y),\quad  y\to \n_1\xi(x,y),\quad y\to \n_2\xi(x,y),
      \quad  y\to \n_j\!\left(\n_k\xi\right)(x,y),
    \end{equation}
    considered as vector fields on $\I_xM$, are $\mathbb R$-linearly
    independent if and only if the
    \begin{equation}
      \label{indep}
      1,\qquad  \frac{\partial \P}{\partial y^1},
      \qquad \frac{\partial \P}{\partial y^2}, \qquad
      \frac{\partial^2 \P}{\partial y^j\partial y^k} - \frac{\lambda}{4}\;g_{jk}
    \end{equation}
    are linearly independent functions on $T_{x}M$.
  \end{lemma}
  Since we assumed that the Finsler function $\F$ is non-Riemannian at
  the point $x$, then $\F^2(x,y)$ is non-quadratic in $y$ and hence the
  function $\P(x,y)$ is non-linear in $y$ on $T_xM$
  (cf.~eq.~(\ref{eq:P})). Let us choose a direction $y_0\!=\!(y_0^1,y_0^2)\in
  T_xM$ with $y_0^1\neq 0$, $y_0^2\neq 0$ and having property that $\P$
  is non-linear $1$-homogeneous function in a conic neighbourhood $U$ of
  $y_0$ in $T_{x}M$. By restricting $U$ if it is necessary we can
  suppose that for any $y\in U$ we have $y^1\neq 0$, $y^2\neq 0$.
  \\
  To avoid confusion between coordinate indexes and exponents, we rename
  the fiber coordinates of vectors belonging to $U$ by
  $(u,v)=(y^1,y^2)$.  Using the values of $\P$ on $U$ we can define a
  $1$-variable function $f=f(t)$ on an interval $I\subset \mathbb R$ by
  \begin{equation}
    \label{eq:P_loc}
    f (t): =\frac{1}{v} \P(x_1,x_2,tv,v).
  \end{equation}
  Then we can express $\P$ and its derivatives with $f$:
  \begin{equation}
    \label{fxy}
    \begin{aligned}
      \P&\!=\!v \, f(u/v),& \ \frac{\partial \P}{\partial y^1}& \!=\!f'(u/v),&
      \ \frac{\partial \P}{\partial y^2}\!=\!f(&u/v)-\frac{u}{v} f'(u/v),
      \\
      \frac{\partial^2 \P}{\partial y^1\partial y^1} &\!=\!\frac{1}{v}
      f''(u/v), & \ \frac{\partial^2 \P}{\partial y^1\partial y^2}&
      \!=\!-\frac{u}{v^2} f''(u/v),& \ \frac{\partial^2 \P}{\partial
        y^2\partial y^2}&\!=\!\frac{u^2}{v^3} f''(u/v).
    \end{aligned}
  \end{equation}

  \begin{lemma}
    \label{1p1p2}
    The functions $1,\frac{\partial \P\,}{\partial y^1},\frac{\partial
      \P\,}{\partial y^2}$ are linearly independent.
  \end{lemma}
  \begin{proof} 
    A nontrivial relation $a+b\frac{\partial \P\,}{\partial
      y^1}+c\frac{\partial \P\,}{\partial y^2}=0$ yields the differential
    equation
    \begin{math}
      a + bf' + c (f-t f')=0.
    \end{math}
    It is clear that both $b$ and $c$ cannot be zero. If $c\neq 0$ we get the
    differential equation
    \[\frac{(a+cf)'}{a+cf}= \frac{1}{t-\frac{b}{c}}.\]
    The solutions is $f(t)=t-(a+b)/c$ and therefore the corresponding
    $\P(u,v)=u-v (a+b)/c$ is linear which is a contradiction.  If $c= 0$, then
    $b\neq0$ and $f=-\frac{a}{b}t+K$.  The corresponding
    $\P(u,v)=-\frac{a}{b}u+Kv$ is again linear which is a contradiction.
  \end{proof}

  \noindent
  Let us assume now, that the infinitesimal holonomy algebra is finite
  dimensional.  We will show that this assumption leads to contradiction which
  will prove then, that the infinitesimal holonomy algebra is actually
  infinite dimensional.
  \\[1ex]
  Since $\I_xM$ is 1-dimensional, according to the Lemma \ref{simultan}, the 4
  vector fields in (\ref{eq:4_vectors}) are linearly dependent for any $j,k\in
  \{1,2\}$. Using Lemma \ref{sec:vect_func} we get that the functions
  \begin{equation}
    \label{eq:4}
    1, \quad \P_1, \quad \P_2, \quad \P_j\P_k-\frac{\lambda}{4} g_{jk}
  \end{equation}
  ($\P_i=\frac{\partial \P}{\partial y^i}$, $\P_{jk}=\frac{\partial^2
    \P}{\partial y^j\partial y^k}$) are linearly dependent for any $j,k\in
  \{1,2\}$.  From Lemma \ref{1p1p2} we know, that the first three functions in
  (\ref{eq:4}) are linearly independent. Therefore by the assumption, the
  fourth function must be a linear combination of the first three, that is
  there exist constants $a_i,b_i,c_i\in \mathbb R$, $i=1,2,3$, such that
  \begin{equation} 
    \begin{aligned}
      \frac{\lambda}{4}\; g_{11}&= \P_1\P_1 + a_1 +b_1 \P_1 + c_1 \P_2, 
      \\
      \frac{\lambda}{4}\; g_{12}&= \P_1\P_2 + a_2 +b_2 \P_1 + c_2 \P_2,
      \\
      \frac{\lambda}{4}\; g_{22}&= \P_2\P_2 + a_3 +b_3 \P_1 + c_3 \P_2.  
    \end{aligned}
  \end{equation}
  Using (\ref{eq:g}) we get $\partial_1 g_{21}-\partial_2 g_{11}=0$ and
  $\partial_1 g_{22}-\partial_2 g_{12}=0$ which yield
  \begin{equation}
    \label{eq:comp_dim_2}
    \begin{aligned}
      \P_2\P_{11} - \P_1\P_{12} + b_2 \P_{11} + (c_2 -b_1) \P_{12} - c_1
      \P_{22}=0,
      \\
      \P_1\P_{22} -\P_2\P_{12} -b_3 \P_{11} +(b_2 - c_3) \P_{12} + c_2 \P_{22}
      =0.
    \end{aligned}
  \end{equation}
  Using the expressions (\ref{fxy}) we obtain from (\ref{eq:comp_dim_2}) the
  equations
  \begin{equation}
    \label{eq:comp_dim_2_f}
    \begin{aligned}
      (f-\frac{u}{v} f')\frac{1}{v} f'' + f'\frac{u}{v^2} f'' + b_2
      \frac{1}{v} f'' - (c_2 -b_1) \frac{u}{v^2} f'' - c_1 \frac{u^2}{v^3}
      f''&=0,
      \\
      f'\frac{u^2}{v^3} f'' +(f-\frac{u}{v} f')\frac{u}{v^2} f'' -b_3
      \frac{1}{v} f'' -(b_2 - c_3) \frac{u}{v^2} f'' + c_2 \frac{u^2}{v^3} f''
      &=0.
    \end{aligned}
  \end{equation}
  Since by the non-linearity of $\P$ on $U$ we have $f''\neq 0$, equations
  (\ref{eq:comp_dim_2_f}) can divide by $f''/v$ and we get
  \begin{equation}
    \begin{aligned}
      f +\; b_2 +\;(b_1-c_2) \frac{u}{v} - \;\frac{c_1u^2}{v^2} &=\;0
      \\
      \frac{u}{v}f -\; b_3 +\;( c_3-b_2) \frac{u}{v} +\; \frac{c_2u^2}{v^2}
      &=\;0.
    \end{aligned}
  \end{equation}
  for any $t=u/v$ in an interval $I\subset\mathbb R$.  The solution of this
  system of quadratic equations for the function $f$ is $f(t) = -c_2\,t-b_2$
  with $c_1 = b_3 = 0$, $b_1 = 2c_2$, $c_3 = 2b_2$.  But this is a
  contradiction, since we supposed that by the non-linearity of $P$ we have
  $f''\neq 0$ on this interval. Hence the functions $1$, $\P_1$, $\P_2$,
  $\P_j\P_k-\frac{\lambda}{4} g_{jk}$ can not be linearly dependent for any
  $j,k\in \{1,2\}$, from which follows the assertion.
\end{proof}

\begin{remark}
  From Proposition \ref{sec:thm1} we get that if $(M,\mathcal F)$ is
  non-Riemannian and $\lambda\neq 0$, then the holonomy group has an infinite
  dimensional tangent algebra. 
\end{remark}
\noindent
Indeed, according to Theorem 6.3 in \cite{MuNa} the infinitesimal holonomy
algebra $\mathfrak{hol}^{*}_x(M)$ is tangent to the holonomy group
$\mathsf{Hol}_x(M)$, from which follows the assertion.
\\[1ex]
Now, we can prove our main result:
\begin{theorem} 
  \label{sec:thm2}
  The holonomy group of a locally projectively flat simply connected Finsler
  manifold $(M,\mathcal F)$ of constant curvature $\lambda$ is finite
  dimensional if and only if $(M,\mathcal F)$ is Riemannian or $\lambda=0$.
\end{theorem}
\begin{proof} 
  If $(M,\mathcal F)$ is Riemannian then its holonomy group is a Lie subgroup
  of the orthogonal group and therefore it is a finite dimensional compact Lie
  group.  If $(M,\mathcal F)$ has zero curvature, then the horizontal
  distribution associated to the canonical connection in the tangent bundle is
  integrable and hence the holonomy group is trivial.
  \\
  If $(M,\mathcal F)$ is non-Riemannian having non-zero curvature $\lambda$,
  then for each tangent $2$-plane $S\!\subset\! T_xM$ the manifold $M$ has a
  totally geodesic submanifold $\widetilde M\!\subset\! M$ such that
  $T_x\widetilde M=S$.  This $\widetilde M$ with the induced metric is a
  locally projectively flat Finsler surface of constant curvature
  $\lambda$. Therefore from Proposition \ref{sec:thm1} we get that
  $\mathfrak{hol}^{*}_x(\widetilde M)$ is infinite dimensional. Moreover,
  according to Theorem 4.3 in \cite{MuNa1}, if a Finsler manifold $(M, \F)$
  has a totally geodesic 2-dimensional submanifold $\widetilde M$ such that
  the infinitesimal holonomy algebra of $\widetilde M$ is infinite
  dimensional, then the infinitesimal holonomy algebra
  $\mathfrak{hol}^{*}_x(M)$ of the containing manifold is also infinite
  dimensional. Using (\ref{eq:dim}) we get that $\mathsf{Hol}_x(M)$ cannot be
  finite dimensional.  Hence the assertion is true.
\end{proof}
We note that there are examples of non-Riemannian type locally projectively
flat Finsler manifolds with $\lambda=0$ curvature, (cf. \cite{shen}).
\begin{remark} 
  In the discussion before the previous theorem, the key condition for the
  Finsler metric tensor was not the positive definiteness but its
  non-degenerate property. Therefore Theorem \ref{sec:thm2} can be generalized
  as follows.
\end{remark}
A pair $(M,\mathcal F)$ is called \emph{semi-Finsler manifold} if in the
definition of Finsler manifolds the positive definitness of the Finsler metric
tensor is replaced by the nondegenerate property. Then we have
\begin{corollary} 
  \label{sec:thm3}
  The holonomy group of a locally projectively flat simply connected
  semi-Finsler manifold $(M,\mathcal F)$ of constant curvature $\lambda$
  is finite dimensional if and only if $(M,\mathcal F)$ is
  semi-Riemannian or $\lambda=0$.
\end{corollary}

\bigskip

\noindent
Zolt\'an Muzsnay
\\
Institute of Mathematics, University of Debrecen,
\\
H-4032 Debrecen, Egyetem t\'er 1, Hungary
\\
{\it E-mail}: {\tt {}muzsnay@science.unideb.hu}\vspace{4mm}
\\
P\'eter T. Nagy
\\
Institute of Applied Mathematics, \'Obuda University
\\
H-1034 Budapest, B\'ecsi \'ut 96/b, Hungary
\\
{\it E-mail}: {\tt {}nagy.peter@nik.uni-obuda.hu}

\end{document}